\documentclass[12pt]{amsart}
\usepackage{amsfonts,amssymb,amscd,amsmath,enumerate,verbatim,calc,latexsym}

\usepackage[dvips]{graphicx}

\def\frk{\mathfrak}               

\def\Phi{{\frk N}}
\def\tb{{\bold t}}

\def\xb{{\bold x}}
\def\yb{{\bold y}}

\def\opn#1#2{\def#1{\operatorname{#2}}} 
\opn\chara{char} 
\opn\length{\ell} 
\opn\pd{pd} 
\opn\rk{rk}
\opn\projdim{proj\,dim} 
\opn\injdim{inj\,dim} 
\opn\rank{rank}
\opn\depth{depth} 
\opn\grade{grade} 
\opn\height{height}
\opn\embdim{emb\,dim} 
\opn\codim{codim}

\opn\Tr{Tr} 
\opn\bigrank{big\,rank}
\opn\superheight{superheight}
\opn\lcm{lcm}
\opn\trdeg{tr\,deg}
\opn\reg{reg} 
\opn\lreg{lreg} 
\opn\ini{in} 
\opn\lpd{lpd}
\opn\size{size}
\opn\mult{mult}
\opn\dist{dist}
\opn\cone{cone}
\opn\lex{lex}
\opn\rev{rev}
\opn\div{div} \opn\Div{Div} \opn\cl{cl} \opn\Cl{Cl}
\opn\Spec{Spec} \opn\Supp{Supp} \opn\supp{supp} \opn\Sing{Sing}
\opn\Ass{Ass} \opn\Min{Min}
\opn\Ann{Ann} \opn\Rad{Rad} \opn\Soc{Soc}
\opn\Syz{Syz} \opn\Im{Im} \opn\Ker{Ker} \opn\Coker{Coker}
\opn\Am{Am} \opn\Hom{Hom} \opn\Tor{Tor} \opn\Ext{Ext}
\opn\End{End} \opn\Aut{Aut} \opn\id{id} \opn\ini{in}

\opn\nat{nat}
\opn\pff{pf}
\opn\Pf{Pf} \opn\GL{GL} \opn\SL{SL} \opn\mod{mod} \opn\ord{ord}
\opn\Gin{Gin}
\opn\Hilb{Hilb}\opn\adeg{adeg}\opn\std{std}\opn\ip{infpt}
\opn\Pol{Pol}
\opn\sat{sat}
\opn\Var{Var}
\opn\Gen{Gen}

\opn\aff{aff} \opn\con{conv} \opn\relint{relint} \opn\st{st}
\opn\lk{lk} \opn\cn{cn} \opn\core{core} \opn\vol{vol}
\opn\link{link} \opn\star{star}
\opn\gr{gr}

\def\Mc{{\mathcal M}}

\def\pot#1#2{#1[\kern-0.28ex[#2]\kern-0.28ex]}

\opn\dirlim{\underrightarrow{\lim}}
\opn\inivlim{\underleftarrow{\lim}}
\let\to=\rightarrow

\def\Implies{\ifmmode\Longrightarrow \else
        \unskip${}\Longrightarrow{}$\ignorespaces\fi}
\def\implies{\ifmmode\Rightarrow \else
        \unskip${}\Rightarrow{}$\ignorespaces\fi}
\def\iff{\ifmmode\Longleftrightarrow \else
        \unskip${}\Longleftrightarrow{}$\ignorespaces\fi}

\let\:=\colon
\newtheorem{Theorem}{Theorem}
\newtheorem{Lemma}[Theorem]{Lemma}

\newtheorem{Example}[Theorem]{Example}

\let\epsilon\varepsilon
\let\phi=\varphi
\let\kappa=\varkappa
\def\qed{\ifhmode\textqed\fi
      \ifmmode\ifinner\quad\qedsymbol\else\dispqed\fi\fi}
\def\textqed{\unskip\nobreak\penalty50
       \hskip2em\hbox{}\nobreak\hfil\qedsymbol
       \parfillskip=0pt \finalhyphendemerits=0}
\def\dispqed{\rlap{\qquad\qedsymbol}}

\opn\dis{dis}
\opn\height{height}
\opn\dist{dist}
\def\pnt{{\raise0.5mm\hbox{\large\bf.}}}

\opn\Lex{Lex}

\begin{document}
\title{
Toric ideals of finite graphs and \\
adjacent $2$-minors 
}
\author{
Hidefumi Ohsugi
and
Takayuki Hibi
}
\thanks{
\\
\, \, \, This research was supported by JST CREST
}
\address{Hidefumi Ohsugi,
Department of Mathematics,
College of Science,
Rikkyo University,
Toshima-ku, Tokyo 171-8501, Japan} 
\email{ohsugi@rikkyo.ac.jp}
\address{Takayuki Hibi,
Department of Pure and Applied Mathematics,
Graduate School of Information Science and Technology,
Osaka University,
Toyonaka, Osaka 560-0043, Japan}
\email{hibi@math.sci.osaka-u.ac.jp}
\begin{abstract}
We study
the problem when an ideal generated by adjacent $2$-minors 
is the toric ideal of a finite graph.
\end{abstract}
\maketitle
Let $X = (x_{ij})_{i=1,\ldots,m\atop j=1,\ldots,n}$ 
be a matrix of $mn$ indeterminates ,
and let $A = K[\{x_{ij}\}_{i=1,\ldots,m\atop j=1,\ldots,n}]$ be
the polynomial ring in $mn$ variables over a field $K$. 
Given 
$1 \leq a_1 < a_2 \leq m$ and $1 \leq b_1 < b_2 \leq n$, 
the symbol $[a_1,a_2|b_1,b_2]$ denotes the $2$-minor
$x_{a_1b_1}x_{a_2b_2} - x_{a_1b_2}x_{a_2b_1}$
of $X$.  In particular $[a_1,a_2|b_1,b_2]$ 
is a binomial of $A$.
A $2$-minor $[a_1,a_2|b_1,b_2]$ of $X$ is {\em adjacent} 
(\cite{HS}) if $a_2 = a_1 + 1$ and $b_2 = b_1 + 1$.
Following \cite{HH},  
we say that a set $\Mc$ of adjacent $2$-minors of $X$
is of {\em chessboard type} if the following conditions 
are satisfied:
\begin{itemize}
\item
if $[a,a+1|b,b+1]$ and $[a,a+1|b',b'+1]$
with $b < b'$ 
belong to $\Mc$,
then $b + 1 < b'$;
\item
if $[a,a+1|b,b+1]$ and $[a',a'+1|b,b+1]$
with $a < a'$ 
belong to $\Mc$,
then $a + 1 < a'$.
\end{itemize}
Given a set $\Mc$ of adjacent $2$-minors of $X$
of chessboard type, 
we introduce the finite graph $\Gamma_\Mc$  
on the vertex set $\Mc$, whose edges are 
$
\{ [a,a+1|b,b+1], [a',a'+1|b',b'+1] \}
$
such that
\begin{itemize}
\item
$[a,a+1|b,b+1] \neq [a',a'+1|b',b'+1]$,
\item
$\{ a, a+1 \} \cap \{ a', a'+1 \} \neq \emptyset$,
\item
$\{ b, b+1 \} \cap \{ b', b'+1 \} \neq \emptyset$.
\end{itemize} 
For example, if
$\Mc = \{ [1,2|2,3], [2,3|3,4], [3,4|2,3], [2,3|1,2]\}$, 
then $\Gamma_\Mc$ is a cycle of length $4$.
The ideal $I_\Mc$ is generated by
$x_{12} x_{23} - x_{13} x_{22}$, 
$x_{23} x_{34} - x_{24} x_{33}$, 
$x_{32} x_{43} - x_{33} x_{42}$ and
$x_{21} x_{32} - x_{22} x_{31}$.
The binomial
$x_{32} (x_{13} x_{21}  x_{34} x_{42} - x_{12} x_{24}  x_{31} x_{43} ) $
belongs to $I_\Mc$ but neither $x_{32}$ nor 
$x_{13} x_{21}  x_{34} x_{42} - x_{12} x_{24}  x_{31} x_{43} $
belongs to $I_\Mc$.
Thus $I_\Mc$ is not prime.

A fundamental fact regarding ideals generated by adjacent $2$-minors is

\begin{Lemma}[\cite{HH}]
\label{HerzogHibi}
Let $\Mc$ be a set of adjacent $2$-minors of $X$,
and let $I_\Mc$ be the ideal of $A$
generated by all $2$-minors
belonging to $\Mc$.
Then, $I_\Mc$ is a prime ideal 
if and only if $\Mc$ is of chessboard type, and
$\Gamma_\Mc$ possesses no cycle of length $4$.
\end{Lemma}

A finite graph $G$ is said to be {\em simple}
if $G$ has no loop and no multiple edge.
Let $G$ be a finite simple graph
on the vertex set
$[d] = \{ 1, \ldots, d\}$,
and let $E(G) = \{ e_1, \ldots, e_n \}$ be its set of edges.
Let $K[\tb] = K[t_1, \ldots, t_d]$ denote the polynomial ring
in $d$ variables over $K$,
and let $K[G]$ denote the subring of $K[\tb]$ generated by the
squarefree quadratic monomials $\tb^e = t_it_j$
with $e = \{ i, j \} \in E(G)$.
The semigroup ring $K[G]$ is called the {\em edge ring}
of $G$.  
Let $K[\yb] = K[y_1, \ldots, y_n]$ denote the polynomial ring
in $n$ variables over $K$.
The kernel $I_G$ of the surjective homomorphism 
$\pi : K[\yb] \to K[G]$ defined by setting
$\pi(y_i) = \tb^{e_i}$ for $i = 1, \ldots, n$
is called the {\em toric ideal} of $G$.  
Clearly, $I_G$ is a prime ideal.
It is known that $I_G$ is generated by
the binomials corresponding to even closed walks of $G$.
See \cite{Vil} , \cite[Chapter 9]{Stu} and
\cite[Lemma 1.1]{OhHiquadratic}
for details.

\begin{Example}
\label{Examplebipartite}
{\em
Let $G$ be a complete bipartite graph with the edge set
$E(G) = \{ \{i, p+ j\} \ | \  1 \leq i \leq p, \ \ 1 \leq j \leq q\}.$
Let $X = (x_{ij})_{i=1,\ldots,p \atop j=1,\ldots,q}$ 
be a matrix of $p q$ indeterminates and
$K[\xb] = K[\{x_{ij}\}_{i=1,\ldots,p \atop j=1,\ldots,q}]$.
Then, $I_G$ is the kernel of the surjective homomorphism 
$\pi : K[\xb] \to K[G]$ defined by setting
$\pi(x_{ij}) = t_i t_{p+j}$ for $1 \leq i \leq p, 1 \leq j \leq q$.
It is known \cite[Proposition 5.4]{Stu} that
$I_G$ is generated by the set of all $2$-minors of $X$.
Note that each 2-minor $x_{i j} x_{i' j'} - x_{i j'} x_{i' j}$ corresponds to
the cycle $\{ \{i,p+j\}, \{p+j, i'\}, \{i',p+j'\}, \{p+j',i\}  \}$ of $G$.
}
\end{Example}

In general, a toric ideal is the defining ideal
of a homogeneous semigroup ring.  
We refer the reader to \cite{Stu}
for detailed information on toric ideals.  
It is known \cite{ES}
that a binomial ideal $I$, i.e., an ideal
generated by binomials, is a prime ideal
if and only if $I$ is a toric ideal.
An interesting research problem on toric ideals is
to determine when a binomial ideal is
the toric ideal of a finite graph.

\begin{Example}
\label{Example}
{\em
The ideal
 $I
=
\langle
x_1x_2 - x_3x_4,
x_1x_2 - x_5x_6,
x_1x_2 - x_7x_8
\rangle$
is the toric ideal of the 
semigroup ring
$
K[t_1t_5,t_2t_3t_4t_5,
t_1t_2t_5,t_3t_4t_5,t_2t_3t_5,t_1t_4t_5,t_1t_3t_5,t_2t_4t_5].
$
If there exists a graph $G$ such that $I = I_G$,
then three quadratic binomials correspond to cycles of length 4.
However, this is impossible
since these three cycles must have common two edges $e_1$ and $e_2$
such that $e_1 \cap e_2 = \emptyset$.
%
Thus, $I$ cannot be the toric ideal of a finite graph.
This observation implies that
the toric ideal of a finite distributive lattice
${\mathcal L}$
(see \cite{HIBI}) is the toric ideal of a finite graph if and only if
 ${\mathcal L}$ is planar.
In fact, if ${\mathcal L}$ is planar,
then it is easy to see that
the toric ideal of ${\mathcal L}$ is the toric ideal of
a bipartite graph.
If ${\mathcal L}$ is not planar,
then ${\mathcal L}$ contains a sublattice that is isomorphic to 
the Boolean lattice $B_3$ of rank 3.
Since the toric ideal of $B_3$ has three binomials above,
the toric ideal of ${\mathcal L}$ cannot be
the toric ideal of a finite graph.
}
\end{Example}

Let $\Mc$ be a set of adjacent $2$-minors.
Now, we determine when a binomial ideal 
$I_\Mc$ generated by $\Mc$
is the toric ideal $I_G$ of a finite graph $G$.
Since $I_G$ is a prime ideal,
according to Lemma \ref{HerzogHibi}, if
there exists a finite graph $G$ with $I_\Mc = I_G$,
then $\Mc$ must be of chessboard type and
$\Gamma_\Mc$ possesses no cycle of length $4$.

\begin{Theorem}
Let $\Mc$ be a set of adjacent $2$-minors.
Then, there exists a finite graph $G$ such that $I_\Mc=I_G$
if and only if 
$\Mc$ is of chessboard type,
$\Gamma_\Mc$ possesses no cycle of length $4$,
and
each connected component of
$\Gamma_\Mc$ possesses at most one cycle.
\end{Theorem}

\begin{proof}
We may assume that $\Mc$ is of chessboard type and
$\Gamma_\Mc$ possesses no cycle of length $4$. 
Let $\Mc = \Mc_1 \cup \cdots \cup \Mc_s$, where
$\Gamma_{\Mc_1}, \ldots, \Gamma_{\Mc_s}$ is
the set of connected components of $\Gamma_\Mc$.
If $i \neq j$, then $f \in \Mc_i$ and $g \in \Mc_j$ have no common variable.
Hence, there exists a finite graph $G$ such that
$I_{\Mc}=I_{G}$
if and only if for each $1 \leq i \leq s$,
there exists a finite graph $G_i$ such that
$I_{\Mc_i}=I_{G_i}$.
Thus, we may assume that $\Gamma_\Mc$ is connected.
Let $p$ be the number of vertices of $\Gamma_\Mc$,
and let $q$ be the number of edges of $\Gamma_\Mc$.
Since $\Gamma_\Mc$ is connected, we have $p \leq q+1$.

\smallskip

\noindent
{\bf [Only if]}
Suppose that there exists a finite graph $G$ with $I_\Mc = I_G$.
From \cite[Theorem 2.3]{HH}, the codimension of $I_\Mc$ is equal to $p$.
Let $d$ be the number of vertices of $G$,
and let $n$ be the number of edges of $G$.
Then, we have $d \leq 4 p -2q$ and $n = 4 p -q$. 
The height of $I_G$ is given in \cite{Vil}.
If $G$ is bipartite, then
the codimension of $I_G$ satisfies
$p \geq n-d+1 \geq  (4 p -q) - (4 p -2q) + 1 = q+1$.
Hence, we have $p = q+1$ and $\Gamma_\Mc$ is a tree.
On the other hand, if $G$ is not bipartite,
then the codimension of $I_G$ satisfies
$p \geq n-d \geq  (4 p -q) - (4 p -2q) = q$.
Hence, we have $p \in \{ q, q+1\}$ and $\Gamma_\Mc$ has at most one cycle.

\smallskip

\noindent
{\bf [If]}
Suppose that $\Gamma_\Mc$ has at most one cycle.
Then, we have $p \in \{ q, q+1\}$. 

\smallskip

\noindent
Case 1. $p = q+1$, i.e., $\Gamma_\Mc$ is a tree.

Through induction on $p$, we will show that there exists a 
connected bipartite graph $G$
such that $I_\Mc = I_G$.
If $p =1 $, then $I_\Mc = I_G$ where $G$ is a cycle of length $4$.
Let $k > 1$, and suppose that the assertion holds for $p = k-1$.
Suppose that $\Gamma_\Mc$ has $k$ vertices.
Since $\Gamma_\Mc$ is a tree,
$\Gamma_\Mc$ has a vertex $v = [a,a+1|b,b+1]$ of degree 1.
Let $\Mc' = \Mc \setminus \{v\} $.
Since $\Gamma_{\Mc'}$ is a tree,
there exists a connected bipartite graph $G'$ such that
$I_{\Mc'}=I_{G'}$
by the hypothesis of induction.
From \cite[Theorem 1.2]{OhHiquadratic},
since $I_{G'}$ is generated by quadratic binomials,
any cycle of $G'$ of length $\geq 6$ has a chord.
Let $v' = [a',a'+1|b',b'+1]$ denote the vertex of $\Gamma_\Mc$ 
that is incident with $v$.
Let $e = \{i,j\}$ be the edge of $G'$ corresponding to
the common variable of $v$ and $v'$.
Let  $\{1,2,\ldots,d\}$ be the vertex set of $G'$.
We now define the connected bipartite graph $G$ on the vertex set $\{1,2,\ldots,d,d+1,d+2\}$
with the edge set $E(G') \cup \{ \{i,d+1\}, \{d+1,d+2\}, \{d+2,j\} \}$.
Then, any cycle of $G$ of length $\geq 6$ has a chord, and hence,
$I_G$ is generated by quadratic binomials.
Thus,
$I_G$ is generated by the quadratic binomials of $I_{G'}$ together with
$v$ corresponding to 
the cycle $\{ \{i,d+1\}, \{d+1,d+2\}, \{d+2,j\} ,\{j,i\}\}$.
Therefore, $I_\Mc = I_G$.

\smallskip

\noindent
Case 2. $p = q$, i.e., $\Gamma_\Mc$ has exactly one cycle.

Then, we have $p \geq 8$.
Through induction on $p$, we will show that there exists a graph $G$
such that $I_\Mc = I_G$.
If $p =8$, then $\Gamma_\Mc$ is a cycle of length 8.
Then,
$I_\Mc = I_G$ where $G$ is the graph shown in Figure 1.

\begin{figure}
  \begin{center}
    \scalebox{0.5}{\includegraphics{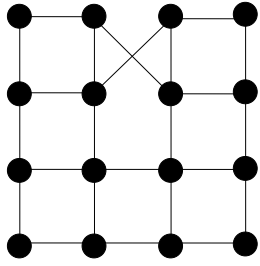}}
  \end{center}
  \caption{Graph for $\Mc$ such that $\Gamma_\Mc$ is a cycle of length 8.}
\end{figure} 

Let $k > 8$ and suppose that the assertion holds for $p = k-1$.
Suppose that $\Gamma_\Mc$ has $k$ vertices.
If $\Gamma_\Mc$ has a vertex $v = [a,a+1|b,b+1]$ of degree 1, then
$\Gamma_{\Mc'}$ where $\Mc' = \Mc \setminus \{v\} $
has exactly one cycle, and hence,
there exists a graph $G'$ such that
$I_{\Mc'}=I_{G'}$
by the hypothesis of induction.
Let $v' = [a',a'+1|b',b'+1]$ denote the vertex of $\Gamma_\Mc$ 
that is incident with $v$.
Let $e = \{i,j\}$ be the edge of $G'$ corresponding to
the common variable of $v$ and $v'$.
Suppose that the vertex set of $G'$ is $\{1,2,\ldots,d\}$.
We now define the graph $G$ on the vertex set $\{1,2,\ldots,d,d+1,d+2\}$
with the edge set $E(G') \cup \{ \{i,d+1\}, \{d+1,d+2\}, \{d+2,j\} \}$.
Since $G'$ satisfies the conditions in \cite[Theorem 1.2]{OhHiquadratic},
it follows that $G$ satisfies the conditions in \cite[Theorem 1.2]{OhHiquadratic}.
Thus, $I_G$ is generated by 
the quadratic binomials of $I_{G'}$ together with
$v$ corresponding to 
the cycle $\{ \{i,d+1\}, \{d+1,d+2\}, \{d+2,j\} ,\{j,i\}\}$.
Therefore, $I_\Mc = I_G$.

Suppose that $\Gamma_\Mc$ has no vertex of degree 1.
Then, $\Gamma_\Mc$ is a cycle of length $k$.
A 2-minor $ad -bc \in \Mc$ is called {\it free}
if one of the following holds:
\begin{itemize}
\item
Neither $a$ nor $d$ appears in other 2-minors of $\Mc$,
\item
Neither $b$ nor $c$ appears in other 2-minors of $\Mc$.
\end{itemize}
From \cite[Lemma 1.6]{HH},
$\Mc$ has at least two free 2-minors.
Let $v = [a,a+1|b,b+1]$ be a free 2-minor of $\Mc$.
We may assume that neither $x_{a, b}$ nor $x_{a+1, b+1}$ appears in other 2-minors of $\Mc$.
Since $\Gamma_\Mc$ is a cycle,
$x_{a+1, b}$ appears in exactly two 2-minors of $\Mc$ and 
$x_{a, b+1}$ appears in exactly two 2-minors of $\Mc$.
Let $\Mc' = \Mc \setminus \{v\} $.
Since $\Gamma_{\Mc'}$ is a tree,
there exists a connected bipartite graph $G'$ such that
$I_{\Mc'}=I_{G'}$
by the argument in Case 1.
Suppose that
the edge $\{1,3\}$ corresponds to the variable $x_{a+1, b}$
and 
the edge $\{2,4\}$ corresponds to the variable $x_{a, b+1}$.
We now define the graph $G$ as shown in Figure 2,
where vertices 1 and  2 belong to the same part of 
the bipartite graph $G'$.
Note that $G$ is not bipartite.
\begin{figure}
  \begin{center}
    \scalebox{0.5}{\includegraphics{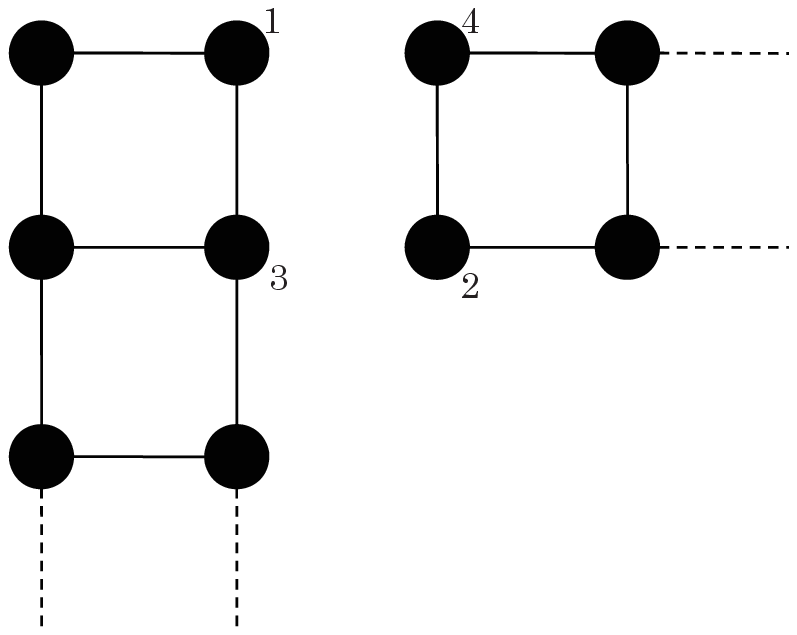}}
    $\begin{array}{c}
\longrightarrow 
\\
\\
\\
\\
\\
\\
\\
\end{array}$
    \scalebox{0.5}{\includegraphics{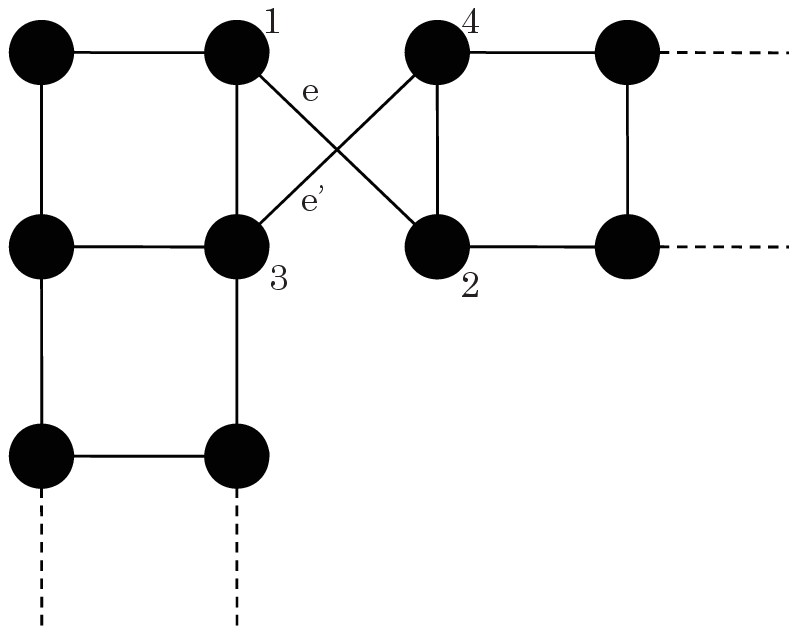}}
    \vspace{-1.5cm}
  \end{center}
\caption{New graph $G$ arising from $G'$.}
\end{figure} 
Let $e = \{1,2\}$ and $e' = \{3,4\}$.
Since $G'$ is a bipartite graph,
it follows that
\begin{itemize}
\item[(a)]
If either $e$ or $e'$ is an edge of an even cycle $C$ of $G$,
then $\{e, e'\} \subset E(C)$.
\item[(b)]
If $C'$ is an odd cycle of $G$, then $\{e, e'\} \cap E(C') \neq \emptyset$.
\end{itemize}

Let $I$ denote the ideal generated by all quadratic binomials in $I_G$.
Since each quadratic binomial in $I_G$ corresponds to a cycle of $G$ of length 4,
it follows that 
$I_\Mc = I$.
Thus, it is sufficient to show that $I_G =I$, i.e., $I_G$ is generated by quadratic binomials.
From \cite[Theorem 1.2]{OhHiquadratic},
since $G'$ is bipartite and since $I_{G'}$ is generated by quadratic binomials,
all cycles of $G'$ of length $\geq 6$ have a chord.

Let $C$ be an even cycle of $G$ of length $\geq 6$.
If $E(C) \cap \{e,e'\} = \emptyset$, then $C$ has an even-chord
since all cycles of the bipartite graph $G'$ of length $\geq 6$ have a chord.
Suppose that 
$\{e, e'\} \subset E(C)$ holds.
Then, either $\{1,3\}$ or $\{2,4\}$ is a chord of $C$.
Moreover, such a chord is an even-chord of $C$
from (b) above.

Let $C$ and $C'$ be odd cycles of $G$ having exactly one common vertex.
From (b) above, we may assume that $e \in E(C) \setminus E(C')$ and 
$e' \in E(C') \setminus E(C)$.
If $\{1,3\}$ does not belong to $E(C) \cup E(C')$, then $\{1,3\}$ satisfies the condition
in \cite[Theorem 1.2 (ii)]{OhHiquadratic}.
If $\{1,3\}$ belongs to $E(C) \cup E(C')$, then $\{2,4\} \notin E(C) \cup E(C')$
since $C$ and $C'$ have exactly one common vertex.
Hence, $\{2,4\}$ 
satisfies the condition in \cite[Theorem 1.2 (ii)]{OhHiquadratic}.

Let $C$ and $C'$ be odd cycles of $G$ having no common vertex.
Then, neither $\{1,3\}$ nor $\{2,4\}$ belong to $E(C) \cup E(C')$.
Hence, $\{1,3\}$ and $\{2,4\}$ satisfy the condition in \cite[Theorem 1.2 (iii)]{OhHiquadratic}.

Thus, from \cite[Theorem 1.2]{OhHiquadratic}, $I_G$ is generated by quadratic binomials.
Therefore, $I_G = I_\Mc$ as desired.
\end{proof}

\end{document}